\documentclass[11pt]{amsart}
\usepackage{graphicx,enumerate,textcomp}
\usepackage{amsmath,amssymb,amsfonts,dsfont,amsthm}
\usepackage{theoremref}
\usepackage{psfrag}
\usepackage{euscript}
\usepackage{mathrsfs}
\usepackage{hyperref}
\usepackage[backend=biber,
style=alphabetic,
sorting=ynt
]{biblatex}
\usepackage[utf8]{inputenc}

\title[Dispersive decay for higher order KdV-type equations]{Dispersive decay bound of small data solutions to higher order scattering-supercritical KdV-type equations}
\author{Jongwon Lee}
\numberwithin{equation}{section}
\begin{document}
\newtheorem{thm}{Theorem}[section]

\newtheorem{cor}[thm]{Corollary}

\newtheorem{lem}[thm]{Lemma}

\newtheorem{prop}[thm]{Proposition}

\theoremstyle{plain}

\theoremstyle{definition}
\newtheorem{rmk}[thm]{Remark}
\newtheorem{defn}[thm]{Definition}
\newtheorem{example}[thm]{Example}

\maketitle
\begin{abstract}
	In this article, we prove that small localized data yield solutions to Higher order Korteweg-de Vries type equation with scattering-supercritical nonlinearity have linear dispersive decay in only a finite length of time. The proof is done by using space-time resonance method and analyzing the oscillatory integrals on the Fourier side.
\end{abstract}
\section{Introduction}
We are going to consider the higher-order KdV-type equation \begin{equation}\label{HO} \begin{cases}\partial_tu-(-\partial_x^2)^{\frac{n-1}{2}}\partial_x^nu=\pm u^{p-1}\partial_xu,\\u(0,x)=u_0(x)\end{cases}
\end{equation} with $n\geq 3$ is an odd integer and $p\geq 2$ is an integer. The equation \eqref{HO} describes the propagation of nonlinear waves in a dispersive medium. In particular, if $n=3$, then \eqref{HO} becomes the KdV equation($p=2$) and its variants($p=3$ for mKdV and $p\geq 4$ for gKdV), and if $n=5$, then it reduces to the Kawahara equation($p=2$) and its variants($p\geq 3$). Also, the seventh-order equation, namely the case when $n=7$, appears in \cite{pomeau1988structural} and \cite{goktacs1997symbolic}.\\The local and global well-posedness of \eqref{HO} can be shown in a standard way introduced in \cite{kenig1993well} and the aid of the conserved quantities $M=\int u^2$ and $H=\int \frac{(\partial_x^k u)^2}{2}\pm\frac{u^{p+1}}{p+1}$ with $k=\frac{n-1}{2}$. Hence the main interest of this article is the small data long-time behavior of \eqref{HO}. The small data asymptotic behavior of the equation \eqref{HO} has been studied extensively. First of all, in the case $p>n,$ Sidi, Sulem, and Sulem proved that a small-data solution of \eqref{HO} scatters to a linear solution after a long time if $p>\frac{\sqrt{n^2+4n}+n}{2}+1$. Kenig, Ponce, and Vega \cite{kenig1991oscillatory} later improved such result to the case $p>\max\{n+1,\frac{n}{2}+3\}$. Therefore, in the case $n\geq 3$ and $p>n+1$, small data linear scattering is known so far. In the case $p=n+1$, Hayashi and Naumkin \cite{hayashi1998large} proved the small data linear scattering in the case $(p,n)=(4,3),$ but the rest still remains as an open problem.\\
The case $p=n$ has also been studied intensively by many authors. First of all, the case $p=n=3$, namely the mKdV equation, was proved by many authors. Deift and Zhou \cite{deift1993steepest} proved the asymptotics of the solutions to mKdV equation using the inverse scattering transform, which depends on the complete integrability of the equation. Hayashi and Naumkin \cite{hayashi1999large} proved that the small data solutions to the mKdV equation shows a modified scattering behavior without using the complete integrability, but their proof required the mean-zero condition of the initial datum. Later the methods requiring neither complete integrability nor mean-zero condition are developed. Germain, Pusateri, and Rousset \cite{germain2016asymptotic} proved the modified scattering by analyzing the oscillatory integrals in the Fourier space. In the same time, Harrop-Griffiths \cite{harrop2016long} proved the same thing using the wave packet testing. In the case $n=p=5$, inspired by \cite{harrop2016long}, Okamoto \cite{okamoto2018long} proved that the fifth-order modified KdV equations show the small data modified scattering similar to the mKdV equation. The same author also proved \cite{okamoto2019asymptotic} the modified scattering behavior in a long time for any $n\geq 4$. Very recently, Wang \cite{wang2021global} proved the modified scattering in the case $n=p=5$ again, but his idea is different with Okamoto; his main argument is adopted from \cite{germain2016asymptotic}. Those results tell that the case $p=n$ is critical in the sense of scattering. On the other hand, when $p=n=2$, where equation \eqref{HO} becomes the Benjamin-Ono equation, Ifrim and Tataru \cite{ifrim2017well} figured out that a small data solution shows a linear dispersive decay in only a finite time.
\\The case $p<n$, which is supercritical in the sense of scattering, is not widely known so far. Ifrim, Koch, and Tataru \cite{ifrim2019dispersive} proved that if $n=3$ and $p=2$, i.e. in the case of KdV equation, a small data solution shows the linear dispersive decay in only a finite time interval, and such time interval cannot be extended further. Adopting the same idea as Ifrim, Koch, and Tataru, the author \cite{lee2022dispersive} proved the similar thing in the case $(n,p)=(5,2)$ and $(5,3)$, namely the Kawahara and modified Kawahara equation, respectively. However, whether the time scale where the dispersive decay bound holds is optimal is still unknown.\\
The above discussions can be summarized as:
\begin{itemize}
    \item When $p>n+1$(subcritical case), the small data solutions to \eqref{HO} scatters to a linear solution.
    \item When $p=n+1$, the long-time behavior of a small data solution is still unknown except the case $n=3$.
    \item When $p=n$(critical case), the small data solutions to \eqref{HO} does not show the linear scattering, but rather a modified scattering behavior if $n=3$, and a finite time linear dispersive decay if $n=2$.
    \item When $p<n$(supercritical case), only the cases $(p,n)=(2,3),\;(2,5),\;(3,5)$ are known. In such cases, a small data solution shows neither linear nor modified scattering behavior. Also it shows the linear dispersive decay in only a finite time.
\end{itemize}
The main goal of this article is to fill the unknown part of the case $p<n$. More specifically, we are going to show that a small data solution to \eqref{HO} with any $(p,n)$ with $p<n$ shows neither linear nor modified scattering, but only a finite-time linear dispersive decay. \\
Here is the main theorem:
\begin{thm} Consider the Cauchy problem
\eqref{HO} where $n\geq 5$ is an odd integer and $p$ is an integer with $2\leq p<n$. Let $u_0$ satisfy the bound \begin{equation}\label{small}
\Vert u_0\Vert_{H^\frac{n-1}{2}}+\Vert xu_0\Vert_{L^2}\leq \epsilon\ll 1.
\end{equation} Then the Cauchy problem \eqref{HO} has a unique global solution $u\in C(\mathbb{R},H^\frac{n-1}{2}(\mathbb{R}))$ which satisfies the linear dispersive decay bound \begin{equation}\label{ddecay}
|D^\beta u(t,x)|\lesssim\epsilon t^{-\frac{1}{n}-\frac{\beta}{n}}\langle t^{-\frac{1}{n}}x\rangle^{-\frac{n-2}{2n-2}+\frac{\beta}{n-1}},\quad 0 \leq \beta \leq n-2,\quad |t|\ll \epsilon^{-\frac{n(p-1)}{n-p}}
\end{equation}
where $D=|\partial_x|=H\partial_x$ with $H$ being the Hilbert transform. In the elliptic region $\{x\gtrsim t^{\frac{1}{n}}\},$ one has the better spatial decay as follows:
\begin{equation}\label{ell}
|D^\beta u(t,x)|\lesssim\epsilon t^{-\frac{1}{n}-\frac{\beta}{n}}\langle t^{-\frac{1}{n}}x\rangle^{-1+\frac{\beta}{n-1}},\quad 0 \leq \beta \leq n-2,\quad |t|\ll \epsilon^{-\frac{n(p-1)}{n-p}}
\end{equation}In both \eqref{ddecay} and \eqref{ell}, if $\beta$ is an integer, then $D^\beta$ can be simply replaced by $\partial_x^\beta$.
\end{thm}
\begin{rmk}
The time scale in which \eqref{ddecay} holds tends to get larger as $p$ increases. This is because the nonlinear term becomes more perturbative as the power gets bigger. On the other hand, the time scale tends to get smaller as $n$ increases. This is because the group velocity for the frequency $\xi$ is given as $-n\xi^{n-1}$, which asserts that the low frequency part travels slower as $n$ increases, which worsens the decay of the solution.
\end{rmk}
\begin{rmk}
Let me give the motivation of the time scale $\epsilon^{-\frac{n(p-1)}{n-p}}$ where the dispersive decay bound holds.\\Suppose that the solution $u$ of \eqref{HO} satisfies the linear decay bound same as \eqref{pointwise}. Then Duhamel's formula, Minkowski's integral inequality, and the bound \eqref{pointwise} gives
\[\Vert u(T)\Vert_{L^2}\lesssim \Vert u(1)\Vert_{L^2}+\int_1^T \epsilon^{p-1}\Vert u(s)\Vert_{L^2}\frac{ds}{s^{\frac{p}{n}}}\lesssim (1+\epsilon^{p-1}T^{1-\frac{p}{n}})\sup_{1\leq t\leq T}\Vert u(t)\Vert_{L^2},\] so if one wants the $L^2$ norm of $u$ to be bounded by $\sup_{1\leq t\leq T}\Vert u(t)\Vert_{L^2}$, we should have $T\lesssim \epsilon^{-\frac{n(p-1)}{n-p}}$. Loosely speaking, one can only guarantee that $u$ behaves `linearly' in such time scale.
\end{rmk}
The proof of the main theorem will be done by the similar argument as in \cite{germain2016asymptotic} and  \cite{wang2021global}.
\\
If $p$ is even, then the equation \eqref{HO} with both $+$ and $-$ sign on the nonlinearity has essentially the same dynamics, since the transformation $u(t,x)\mapsto -u(t,x)$ interchanges the two cases. On the other hand, two cases differ if $p$ is odd. However, such difference is not significant in the small data analysis, so we may only consider the equation with $+$ sign.
\section{Outline of the proof}
Before joining the proof of the main theorem, we first give a dispersive decay bound of linear solutions, which is a generalization of Lemma 2.1 of \cite{germain2016asymptotic} and \cite{wang2021global}:
\begin{lem}
Let $S(t)$ be the linear propagator of the equation \[\partial_tu+(-1)^{\frac{n+1}{2}}\partial_x^nu=0.\]For any $t\geq 1$, $x\in\mathbb{R}$, an odd integer $n$ with $n\geq 3$, and $u(t,x)=S(t)f(t,x)$ with $f$ satisfying \begin{equation}\label{fbd}\Vert \hat{f}(t)\Vert_{L^\infty}+t^{-\frac{1}{2n}}\Vert xf(t)\Vert_{L^2}\leq A,\end{equation} the following estimate holds:\begin{equation}\label{pointwise}
\left||\partial_x|^\beta u(x,t)\right|\lesssim At^{-\frac{1}{n}-\frac{\beta}{n}}\langle t^{-\frac{1}{n}}x\rangle^{-\frac{n-2}{2n-2}+\frac{\beta}{n-1}},\quad 0\leq \beta\leq n-2.
\end{equation} Furthermore, in the region $x>0$, one has the better spatial decay:\begin{equation}\label{betterdecay}\left||\partial_x|^\beta u(x,t)\right|\lesssim At^{-\frac{1}{n}-\frac{\beta}{n}}\langle t^{-\frac{1}{n}}x\rangle^{-1+\frac{\beta}{n-1}},\quad 0\leq \beta\leq n-2.
\end{equation} In both \eqref{pointwise} and \eqref{betterdecay}, if $\beta$ is an integer, $|\partial_x|$ can be replaced by $\partial_x$. Moreover, under the same assumption, for any $\beta\in [0,\frac{n-2}{2})$ and all $p$ satisfying $p(\frac{n-2}{2n-2}-\frac{\beta}{n-1})>1,$\begin{equation}\label{lp}
    \Vert |\partial_x|^\beta u(t)\Vert_{L^p}\lesssim At^{-\frac{1}{n}-\frac{\beta}{n}+\frac{1}{np}}.
\end{equation}
\end{lem}
 Since the arguments in \cite{germain2016asymptotic} and \cite{wang2021global} are still valid here, we skip the details.
\\To give the nonlinear analysis, we first define the norm:\begin{align*}\Vert u\Vert_X:=&\Vert u(t)\Vert_{H^{\frac{n-1}{2}}}+t^{-\frac{1}{2n}}\Vert xf(t)\Vert_{L^2}+\Vert \hat{f}(t)\Vert_{L^\infty}.\end{align*}
First note that this norm is strong enough to guarantee the bound \eqref{fbd}. The $H^{\frac{n-1}{2}}$ norm of $u$ is simply the scale where the equation \eqref{HO} is globally well-posed(at least for the small data). The terms $t^{-\frac{1}{2n}}\Vert xf(t)\Vert_{L^2}+\Vert \hat{f}(t)\Vert_{L^\infty}$ are simply the norm needed in \eqref{fbd}.
Now we have a bootstrap assumption \begin{equation}\label{bootstrap}\Vert u\Vert_X\leq M\epsilon,\quad |t|<(M^{\frac{p+1}{p-1}}\epsilon)^{-\frac{n(p-1)}{n-p}}\end{equation} where $M\gg 1$ is a sufficiently large universal constant not depending on $\epsilon$ and $\epsilon$ is sufficiently small so that $M\epsilon\ll 1$. Here the bound on $t$ is where the time bound in \eqref{ddecay} comes from. The bound \eqref{bootstrap} and Lemma 2.1 readily yield the bound \begin{align*}
\left||\partial_x|^\beta u(x,t)\right|\lesssim M\epsilon t^{-\frac{1}{n}-\frac{\beta}{n}}\langle t^{-\frac{1}{n}}x\rangle^{-\frac{n-2}{2n-2}+\frac{\beta}{n-1}},\quad 0\leq \beta\leq n-2.
\end{align*} Our goal is to show that \begin{equation}\label{nts}
    \Vert u\Vert_X\lesssim\epsilon,\quad |t|<(M^{\frac{p+1}{p-1}}\epsilon)^{-\frac{n(p-1)}{n-p}}.
\end{equation} The proof of \eqref{nts} is divided into two parts. First is the energy estimates, which proves the bound on $\Vert u(t)\Vert_{H^{\frac{n-1}{2}}}+t^{-\frac{1}{2n}}\Vert xf(t)\Vert_{L^2}$. Second is the frequency analysis of oscillatory integrals, which is necessary to prove the bound on $\Vert \hat{f}(t)\Vert_{L^\infty}.$\\For notational convenience, we will often abbreviate $M\epsilon:=\epsilon_1.$ From the time bound we have just assumed, it is easy to see that $\epsilon_1^p<\epsilon t^{-\frac{n-p}{n}}$ and $\epsilon_1^{p+1}<\epsilon^2 t^{-\frac{n-p}{n}}$, which will be used frequently in the later analysis.
\section{Energy Estimates}
The goal of this section is to prove the following proposition:
\begin{prop}
Let $u$ be a solution to \eqref{HO} satisfying \eqref{small} and \eqref{bootstrap}. Then one has \begin{equation}\label{energy}
\Vert u(t)\Vert_{H^\frac{n-1}{2}}+t^{-\frac{1}{2n}}\Vert xf(t)\Vert_{L^2}\lesssim\epsilon,\quad |t|<(M^\frac{p+1}{p-1}\epsilon)^{-\frac{n(p-1)}{n-p}}
\end{equation} for any $p$.
\end{prop}
First prove the bound \eqref{energy}. The bound of $\Vert u(t)\Vert_{H^\frac{n-1}{2}}$ is obvious due to the conservation of mass and Hamiltonian of \eqref{HO}.\\To prove the bound on $\Vert xf(t)\Vert_{L^2}$, first note that the equation \eqref{HO} enjoys a scaling symmetry $u(t,x)\mapsto \lambda^{\frac{n-1}{p-1}} u(\lambda^n t,\lambda x)$, so the scaling vector field is given as $\mathcal{S}=nt\partial_t+x\partial_x+\frac{n-1}{p-1}$. Let $\mathcal{L}$ be the linear part of \eqref{HO}. Using the commutator relations \[[\mathcal{L},\mathcal{S}]=n\mathcal{L},\quad [\mathcal{S},\partial_x]=-\partial_x,\] one may see that $z:=\mathcal{S}u$ satisfies the equation \[\partial_tz+(-1)^{\frac{n+1}{2}}\partial_x^nz=\partial_x(u^{p-1}z),\]and $w:=\partial_x^{-1}z$ (formally) satisfies the linearized equation \begin{equation}\label{linearized}\partial_tw+(-1)^{\frac{n+1}{2}}\partial_x^nw=u^{p-1}\partial_xw.
\end{equation}Here why I said `formally' is that $z$ may not be well-defined, since $z=(x+nt\partial_x^{n-1})u+\frac{n}{p}tu^p+\frac{n-p}{p-1}\partial_x^{-1}u$ contains an inverse derivative of $u$. Hence, we shall remove the inverse derivative and define $y:=(x+nt\partial_x^{n-1})u+\frac{n}{p}tu^p=S(t)(xf)+\frac{n}{p}tu^p.$ Then $y$ satisfies the inhomogeneous linearized equation \[\partial_ty+(-1)^{\frac{n+1}{2}}\partial_x^ny=u^{p-1}\partial_xy-\frac{n-p}{p}u^p.\]
We shall first prove the bound on $y$, and finally the bound on $f=S(-t)u$.\\
First by a direct energy estimate one has
\begin{align*}
\frac{1}{2}\frac{d}{dt}\Vert y\Vert_{L^2}^2 =& \int y(u^{p-1}\partial_xy-\frac{n-p}{p}u^p)dx=-\frac{1}{2}\int \partial_x(u^{p-1})y^2dx-\frac{n-p}{p}u^pydx\\\lesssim& \Vert \partial_x(u^{p-1})\Vert_{L^\infty}\Vert y\Vert_{L^2}^2+\Vert u^p\Vert_{L^2}\Vert y\Vert_{L^2}\lesssim\epsilon_1^{p-1}t^{-\frac{p}{n}}\Vert y\Vert_{L^2}^2+\epsilon_1^{1-p}t^{\frac{p}{n}}\Vert u\Vert_{2p}^{2p}\\\lesssim& \epsilon_1^{p-1}t^{-\frac{p}{n}}\Vert y\Vert_{L^2}^2+\epsilon_1^{p+1}t^{-\frac{p-1}{n}}.
\end{align*}
Now multiplying $\exp(C\epsilon_1^{p-1}t^{1-\frac{p}{n}})$ on both sides and integrating in $t$ gives \[\Vert y\Vert_{L^2}^2\lesssim \epsilon_1^{p+1}t^{1-\frac{p-1}{n}}\leq \epsilon^2t^{\frac{1}{n}},\quad |t|<(M^\frac{p+1}{p-1}\epsilon)^{-\frac{n(p-1)}{n-p}}\] so one has $\Vert y\Vert_{L^2}\lesssim \epsilon t^{\frac{1}{2n}}$. Then a bound on $xf$ is given as \[\Vert xf\Vert_{L^2}\lesssim \Vert y\Vert_{L^2}+t\Vert u\Vert_{2p}^p\lesssim \epsilon t^{\frac{1}{2n}}+\epsilon_1^p t^{-\frac{p}{n}+\frac{1}{2n}+1}\lesssim \epsilon t^{\frac{1}{2n}}.\]
\section{Frequency Analysis of Oscillatory Integrals}
In this section, we are going to prove that $\Vert \hat{f}(t)\Vert_{L^\infty}\lesssim\epsilon$ under the bootstrap assumption \eqref{bootstrap}. To prove this, we have to prove the following:
\begin{prop}
Let $u$ be a solution to \eqref{HO} satisfying \eqref{small} and \eqref{bootstrap}, and $f(t)=S(-t)u(t).$ Then $\hat{f}$ satisfies the equation \begin{equation}\label{hatf}
\partial_t\hat{f}(t,\xi)=\frac{t^{-\frac{p-1}{2}}\xi}{|\xi|^{\frac{(p-1)(n-2)}{2}}}\sum_{\substack{0\leq j\leq p-1\\j\neq \frac{p-2}{2}}}c_je^{id_j t\xi^n}\hat{f}(t,-\frac{\xi}{p-2j-2})^{j+1}\hat{f}(t,\frac{\xi}{p-2j-2})^{p-j-1}\mathbf{1}_{|\xi|>t^{-\frac{1}{n}}}+R(t,\xi),
\end{equation} for some constants $c_j\in\mathbb{C}$ and $d_j\in\mathbb{R}$ in the time scale $|t|\ll_M \epsilon^{-\frac{n(p-1)}{n-p}}$. In particular, if $p$ is odd and $j=\frac{p-1}{2}$ or $\frac{p-3}{2}$, then $c_j$ is purely imaginary and $d_j=0$. Moreover, the remainder term $R(t,\xi)$ satisfies \begin{equation}\label{errorbd}
\Vert R(\cdot,\xi)\Vert_{L_t^1}\lesssim\epsilon
\end{equation} uniformly in $\xi$.
\end{prop}
\begin{rmk}
The case when $n=p=3$ and $n=p=5$ are exactly mentioned in \cite{germain2016asymptotic} and \cite{wang2021global}, respectively.
\end{rmk}
First note that $f$ satisfies the equation \begin{align*}\partial_t\hat{f}(t,\xi)=i\xi\int e^{-itH}\hat{f}(t,\xi_1)\cdots\hat{f}(t,\xi_p)d\xi_1\cdots d\xi_{p-1}\\\end{align*} with $H:=H(\xi,\xi_1,\ldots,\xi_{p-1})=\xi^n-\xi_1^n-\cdots-\xi_{p}^n$, $k$ being a real constant, and $\xi_p=\xi-\xi_1-\cdots-\xi_p.$ Then the stationary points of the integral on the right hand side, namely the solutions of $\nabla H_{(\xi_1,\ldots,\xi_{p-1})}=0,$ are given as the permutations of the following points:\begin{align*}
    \xi_{(j)}=(\overbrace{-\xi_{j,p},\ldots,-\xi_{j,p}}^{j \text{ entries}},\overbrace{\xi_{j,p},\ldots,\xi_{j,p}}^{p-j-1\text{ entries}}),\quad 0\leq j\leq p-1,\quad \xi_{j,p}=\frac{\xi}{p-2j-2}
\end{align*} provided $j\neq \frac{p-2}{2}.$ Now the stationary phase argument yields the principal terms of \eqref{hatf}.\\
Also note that there is no stationary point with $j= \frac{p-2}{2}.$ Also a direct calculation shows that \[H(\xi,\xi_{(j)})=\left(1-\frac{1}{(p-2j-2)^{n-1}}\right)\xi^n,\] so the points $\xi_{(j)}$ (and their permutations) are time-resonant if and only if  $p-2j-2=\pm1,$ or equivalently $j=\frac{p-1}{2},\;\frac{p-3}{2}.$ Hence the space-time resonance set is nonempty only when $p$ is odd. More precisely, if $p$ is odd,\begin{equation}
    \partial_tf(t,\xi)=ic\frac{t^{-\frac{p-1}{2}}\xi}{|\xi|^{\frac{(p-1)(n-2)}{2}}}|\hat{f}(t,\xi)|^{p-1}\hat{f}(t,\xi)\mathbf{1}_{|\xi|>t^{-\frac{1}{n}}}+(\text{oscillatory terms})+ R(t,\xi),
\end{equation} where `oscillatory terms' are the terms in \eqref{hatf} with $j\neq \frac{p-1}{2},\frac{p-3}{2}.$ Here it is not trivial that $c$ is indeed a real number. This fact is a direct consequence of the lemma below and the stationary phase formula.
\begin{lem}
$\mathrm{sign}\mathrm{Hess}H(\xi,\xi_{(j)})=0$ when $p$ is odd and $j=\frac{p-1}{2},\;\frac{p-3}{2}.$
\end{lem}
\begin{proof}
A direct calculation shows that when $j=\frac{p-1}{2},$
\begin{equation}
    M_1:=\mathrm{Hess}(\xi,\xi_{(j)})=-n(n-1)\xi^{n-2}\begin{pmatrix}2&1&\cdots&&&1\\1&\ddots&&&&\vdots\\\vdots&&2&&&\\&&&0&&\\&&&&\ddots&1\\1&\cdots&&&1&0\end{pmatrix}\end{equation} where the number of 2's and 0's are both $\frac{p-1}{2}$ and off-diagonal components are all 1. Similarly, when $j=\frac{p-3}{2},$
\begin{equation}
    M_2:=\mathrm{Hess}(\xi,\xi_{(j)})=n(n-1)\xi^{n-2}\begin{pmatrix}0&1&\cdots&&&1\\1&\ddots&&&&\vdots\\\vdots&&0&&&\\&&&2&&\\&&&&\ddots&1\\1&\cdots&&&1&2\end{pmatrix}\end{equation} where the number of 0's is $\frac{p+1}{2}$ and the number of 2's is $\frac{p-3}{2}.$
Now find the eigenvalues of $M_1$. First note that the upper $\frac{p-1}{2}$ rows of $M_1+I$ are identical, so the equation $(M_1+I)x=0$ has trivial solutions and $-1$ is an eigenvalue of $M_1$ with the eigenspace \[E_1=\{(\xi_1,\ldots,\xi_{\frac{p-1}{2}},0,\ldots,0)\in\mathbb{R}^{p-1}:\sum_{j=1}^{\frac{p-1}{2}}\xi_j=0\}.\] Hence the eigenvalue $-1$ has multiplicity\footnote{Strictly speaking, one should mention whether the multiplicity is algebraic or geometric. However, since we are dealing with symmetric matrices, the two are the same.} $\frac{p-3}{2}$. In the similar way, one may find that $1$ is the eigenvalue of $M_1$ with multiplicity $\frac{p-3}{2}$.\\Now suppose $t\neq\pm 1.$ Then the solution of $(M_1-tI)x=0$ must satisfy $x_j=x_k$ whenever $1\leq j,k\leq \frac{p-1}{2}$ or $\frac{p+1}{2}\leq j,k\leq p-1.$ Thus we may let $x_j=a$ for $1\leq j\leq \frac{p-1}{2}$ and $x_j=b$ for $\frac{p+1}{2}\leq j\leq p-1.$ Then the equation reduces to the following linear equation with two variables:\[\begin{cases}(\frac{p+1}{2}-t)a+\frac{p-1}{2}b=0,\\\frac{p-1}{2}a+(\frac{p-3}{2}-t)b=0.\end{cases}\]To guarantee that this equation has a nontrivial solution, we must have \[(\frac{p+1}{2}-t)(\frac{p-3}{2}-t)-(\frac{p-1}{2})^2=0,\] which is equivalent to $t^2-(p-1)t-1=0.$ Since this quadratic equation has two real roots with opposite sign, we may finally conclude that $M_1$ has the same number of positive and negative eigenvalues(counting with multiplicity), so $\mathrm{sign}M_1=0.$ \\Since finding eigenvalues of $M_2$ is completely similar, we omit the details.
\end{proof}

\begin{proof}[Proof of \eqref{nts} under Proposition 4.1] Let \[\hat{w}(t,\xi)=e^{-iB(t,\xi)}\hat{f}(t,\xi),\quad B(t,\xi)=\frac{c\xi}{|\xi|^{\frac{(p-1)(n-2)}{2}}}\int_1^t |\hat{f}(s,\xi)|^{p-1}\frac{ds}{s^{\frac{p-1}{2}}}.\] If $p$ is even, then we may simply let $c=0$, so such transformation is nontrivial only when $p$ is odd. Observe that $\hat{w}$ satisfies the equation
\begin{align*}\partial_t \hat{w}(t,\xi)=e^{-iB(t,\xi)}[\frac{t^{-\frac{p-1}{2}}\xi}{|\xi|^{\frac{(p-1)(n-2)}{2}}}\sum c_je^{id_j t\xi^n}\hat{f}(t,-\xi_{j,p})^{j+1}\hat{f}(t,\xi_{j,p})^{p-j-1}\mathbf{1}_{|\xi|>t^{-\frac{1}{n}}}+R(t,\xi)]\end{align*} where the summation is taken over a finite number of $j$'s. Hence one has \begin{align*}
    &|\hat{f}(t,\xi)|\lesssim |\hat{u}_0(\xi)|\\+ &\frac{1}{|\xi|^{\frac{(p-1)(n-2)}{2}-1}}\sum \left|\int_{|\xi|^{-n}}^t e^{id_js\xi^n}\hat{f}(s,-\xi_{j,p})^{j+1}\hat{f}(s,\xi_{j,p})^{p-j-1}\frac{ds}{s^{\frac{p-1}{2}}}\right|+\int |R(s,\xi)|ds\\\lesssim& \epsilon+\frac{1}{|\xi|^{\frac{(p-1)(n-2)}{2}-1}}\sum \left|\int_{|\xi|^{-n}}^t e^{id_js\xi^n}e^{-iB(s,\xi)}\hat{f}(s,-\xi_{j,p})^{j+1}\hat{f}(s,\xi_{j,p})^{p-j-1}\frac{ds}{s^{\frac{p-1}{2}}}\right|.
\end{align*} Hence it is enough to show that \begin{equation}\label{timeint}
|\xi|^{1-\frac{(p-1)(n-2)}{2}}\left|\int_{|\xi|^{-n}}^t e^{id_js\xi^n}\hat{f}(s,-\xi_{j,p})^{j+1}\hat{f}(s,\xi_{j,p})^{p-j-1}\frac{ds}{s^{\frac{p-1}{2}}}\right|\lesssim\epsilon
\end{equation} for each $j$.\\
If $p>2$, a standard integration by parts gives
\begin{align*}
&\int_{|\xi|^{-n}}^t e^{id_js\xi^n}e^{-iB(s,\xi)}\hat{f}(s,-\xi_{j,p})^{j+1}\hat{f}(s,\xi_{j,p})^{p-j-1}\frac{ds}{s^{\frac{p-1}{2}}}\\=&\frac{j+1}{id_j\xi^n}\int_{|\xi|^{-n}}^te^{id_js\xi^n}e^{-iB(s,\xi)}\partial_s\hat{f}(s,-\xi_{j,p})\hat{f}(s,-\xi_{j,p})^{j}\hat{f}(s,\xi_{j,p})^{p-j-1}\frac{ds}{s^{\frac{p-1}{2}}}\\-&\frac{p-j-1}{id_j\xi^n}\int_{|\xi|^{-n}}^te^{id_js\xi^n}e^{-iB(s,\xi)}\partial_s\hat{f}(s,\xi_{j,p})\hat{f}(s,-\xi_{j,p})^{j+1}\hat{f}(s,\xi_{j,p})^{p-j-2}\frac{ds}{s^{\frac{p-1}{2}}}\\+&\frac{p-1}{2id_j\xi^n}\int_{|\xi|^{-n}}^te^{id_js\xi^n}e^{-iB(s,\xi)}\hat{f}(s,-\xi_{j,p})^{j+1}\hat{f}(s,\xi_{j,p})^{p-j-1}\frac{ds}{s^{\frac{p-1}{2}}}\\+&\frac{1}{id_j\xi^n}\int_{|\xi|^{-n}}^te^{id_js\xi^n}i\partial_sB(s,\xi)e^{-iB(s,\xi)}\hat{f}(s,-\xi_{j,p})^{j+1}\hat{f}(s,\xi_{j,p})^{p-j-1}\frac{ds}{s^{\frac{p+1}{2}}}\\+&\frac{1}{id_j\xi^n}e^{id_js\xi^n}e^{-iB(s,\xi)}\hat{f}(s,-\xi_{j,p})^{j+1}\hat{f}(s,\xi_{j,p})^{p-j-1}\frac{1}{s^{\frac{p-1}{2}}}\bigg|_{s=|\xi|^{-n}}^{s=t}:=J_1+J_2+J_3+J_4+J_5.
\end{align*}
Here $J_1$ and $J_2$ can be estimated in the similar way, so we only estimate $J_1$.
\begin{align*}
&|\xi|^{1-\frac{(p-1)(n-2)}{2}}|J_1|\\\lesssim &\frac{\epsilon_1^{p-1}}{|\xi|^{n+\frac{(p-1)(n-2)}{2}-1}}\int_{|\xi|^{-n}}^t|\frac{\xi}{s^{\frac{p-1}{2}}|\xi|^{\frac{(p-1)(n-2)}{2}}}\sum c_je^{id_jt\xi^n}\hat{f}(t,-\xi_{j,p})^{j+1}\hat{f}(t,\xi_{j,p})^{p-j-1}+R(s,\xi)|\frac{ds}{s^{\frac{p-1}{2}}}\\\lesssim&\frac{\epsilon_1^{2p-1}}{|\xi|^{n+(p-1)(n-2)-2}}\int_{|\xi|^{-n}}^t\frac{ds}{s^{p-1}}+\frac{\epsilon_1^{p-1}}{|\xi|^{n+\frac{(p-1)(n-2)}{2}-1}}\int_{|\xi|^{-n}}^t|R(s,-\xi_{j,p})|\frac{ds}{s^{\frac{p-1}{2}}}\\\lesssim& \epsilon_1^{2p-1}|\xi|^{2p-2n}+\epsilon\epsilon_1^{p-1}|\xi|^{p-n}\leq\epsilon_1^{2p-1}t^{2(1-\frac{p}{n})}+\epsilon\epsilon_1^{p-1}t^{1-\frac{p}{n}}\lesssim \epsilon,
\end{align*} by \eqref{bootstrap}, \eqref{errorbd}. \\$J_3$ can be estimated as follows:
\begin{align*}
|\xi|^{1-\frac{(p-1)(n-2)}{2}}|J_3|\lesssim \frac{\epsilon_1^p}{|\xi|^{n+\frac{(p-1)(n-2)}{2}-1}}\int_{|\xi|^{-n}}^t \frac{ds}{s^{\frac{p+1}{2}}}\lesssim \epsilon_1^p |\xi|^{p-n}\lesssim \epsilon.
\end{align*} Similarly,
\begin{align*}
|\xi|^{1-\frac{(p-1)(n-2)}{2}}|J_4|\lesssim &\frac{\epsilon_1^p}{|\xi|^{n+(p-1)(n-2)-2}}\int_{|\xi|^{-n}}^t|\hat{f}(s,\xi)|^{p-1}\frac{ds}{s^{p-1}}\lesssim\frac{\epsilon_1^{2p-1}}{|\xi|^{n+(p-1)(n-2)-n(p-2)}}\\=&\epsilon_1^{2p-1}|\xi|^{2(p-n)}\lesssim \epsilon
\end{align*} for $p>2$, 
\\The bound for $J_5$ can be obtained in the similar way.\\
If $p=2$, then the stationary point emerges only when $j=1$. In that case \eqref{timeint} can be estimated directly, namely \begin{align*}
|\xi|^{\frac{4-n}{2}}\left|\int_{|\xi|^{-n}}^t e^{id_js\xi^n}\hat{f}(s,-\xi_{1,2})^{2}\frac{ds}{s^{\frac{1}{2}}}\right|\lesssim |\xi|^{\frac{4-n}{2}}\epsilon_1^2t^{\frac{1}{2}}\lesssim \epsilon_1^2 t^{1-\frac{2}{n}}\leq \epsilon.
\end{align*} This completes the proof of \eqref{nts} under Proposition 4.1.
\end{proof}
\begin{rmk}
In the case of $p=2$ and $n=3$, namely the KdV equation, one cannot apply the same argument as above, since we cannot bound $|\xi|^{\frac{1}{2}}$ by a proper time power under the assumption $|\xi|>t^{-\frac{1}{3}}.$ This is why Ifrim, Koch, and Tataru \cite{ifrim2019dispersive} had to adopt a totally different approach to prove the linear dispersive decay bound of small data solutions to KdV equations.
\end{rmk}
Now it remains to prove Proposition 4.1. To prove it, we need the following lemmas, which are the higher-dimensional generalizations of Lemma A.1 and A.2 in \cite{germain2016asymptotic}:
\begin{lem}[Stationary phase in any dimension] Let $\chi\in C_0^\infty(\mathbb{R}^d)$ be supported inside the ball $B(0,2)$ and $|\nabla \chi|+|\nabla^2\chi|\lesssim 1.$ Also let $\psi\in C^\infty(\mathbb{R}^d)$ satisfy $|\det\mathrm{Hess}\psi|\gtrsim 1$ and $|\nabla\psi|+|\nabla^2\psi|+|\nabla^3\psi|\lesssim 1$ on the support of $\chi$. Let \[I=\int_{\mathbb{R}^d} e^{i\lambda \psi(\eta)}F(\eta)\chi(\eta)d\eta.\] Then, for any $\alpha\in [0,1]$,
\begin{enumerate}[(i)]
    \item If $\nabla\psi$ only vanishes at $\eta=\eta_0,$
    \[I=\frac{2\pi e^{i\frac{\pi}{4}s}}{\sqrt{\Delta}}\frac{e^{i\lambda \psi(\eta_0)}}{\lambda^{\frac{d}{2}}}F(\eta_0)\chi(\eta_0)+O\left(\frac{\Vert \langle x\rangle^{2\alpha}\hat{F}\Vert_{L^1}}{\lambda^{\frac{d}{2}+\alpha}}\right),\] where $s=\mathrm{sign}\mathrm{Hess}\psi$ and $\Delta=|\det\mathrm{Hess}\psi|.$
    \item If $|\nabla\psi|\gtrsim 1$,
    \[I=O\left(\frac{\Vert \langle x\rangle^{\alpha}\hat{F}\Vert_{L^1}}{\lambda^{\frac{d}{2}+\alpha}}\right).\]
\end{enumerate}
\end{lem}
\begin{lem}[Boundedness of pseudo-product operators] Assume that $m\in L^1(\mathbb{R}^d)$ satisfy \[\Vert \int_{\mathbb{R}^d}m(\eta)e^{ix\cdot\eta}d\eta\Vert_{L_x^1}\leq A\] for some $A>0$. Then, for any $p_j\in [1,\infty]$ with $1\leq j\leq d+1$ and $\sum_{j}p_j^{-1}=1,$ one has \[\left|\int_{\mathbb{R}^d}m(\eta)\left[\prod_{j=1}^d \hat{f}_j(\eta_j)\right]\hat{f}_{d+1}(-\sum_{j=1}^d\eta_j)d\eta\right|\lesssim A\prod_{j=1}^{d+1}\Vert f_j\Vert_{L^{p_j}}.\]
\end{lem}

\begin{lem}
Assume that $f$ satisfy \eqref{bootstrap}. Then for any $\alpha\in[0,\frac{1}{2}]$ we have the following estimates:
\begin{align}\label{freqloc}
\begin{split}
\Vert \hat{f_j}\Vert_{L^\infty}&\lesssim\epsilon_1,\\\Vert xf_j\Vert_{L^2}&\lesssim [2^{-\frac{j}{2}}+t^{\frac{1}{2n}}\min(1,(2^jt^{\frac{1}{n}})^{-\alpha})]\epsilon_1,\\\Vert f_j\Vert_{L^1}&\lesssim (1+(2^jt^{\frac{1}{n}})^{\frac{1}{4}})\epsilon_1,\\\Vert f_j\Vert_{L^1}&\lesssim(1+(2^jt^{\frac{1}{n}})^{\frac{1}{4}-\frac{\alpha}{2}})\epsilon_1,\\\Vert|x|^{2\rho}f_j\Vert_{L^1}&\lesssim\Vert f_j\Vert_{L^2}^{\frac{1}{2}-2\rho}\Vert xf_j\Vert_{L^2}^{\frac{1}{2}+2\rho}\lesssim 2^{-2\rho j}(1+(2^jt^{\frac{1}{n}})^{\frac{1}{4}+\rho})\epsilon_1,\quad 0\leq \rho<\frac{1}{4}.
\end{split}
\end{align}
Moreover, if $2^j\geq t^{-\frac{1}{n}},$ then $f_{<j}$ satisfies
\begin{align}\label{highfreq}
\begin{split}
\Vert f_{<j}\Vert_{L^1}&\lesssim (2^jt^{\frac{1}{n}})^{\frac{1}{4}-\frac{\alpha}{2}}\epsilon_1,\\\Vert|x|^{2\rho}f_{<j}\Vert_{L^1}&\lesssim 2^{-2\rho j}(2^jt^{\frac{1}{n}})^{\frac{1}{4}+\rho}\epsilon_1,\quad 0\leq \rho<\frac{1}{4}.
\end{split}
\end{align}
\end{lem}
\begin{proof} The argument in Lemma A.1, A.2, and the beginning of Section 2.4 in \cite{germain2016asymptotic} still works, so we omit the details.\end{proof}
Now we are ready to prove Proposition 4.1. To prove this, first we decompose the integral as follows:
\begin{align*}
&i\xi\int e^{-itH}\hat{f}(t,\xi_1)\cdots\hat{f}(t,\xi_p)d\xi_1\cdots d\xi_{p-1}\\=&\sum_{(k_1,\ldots,k_{p-1})\in \mathbb{Z}^{p-1}}i\xi\int e^{-itH}\hat{f}(t,\xi_1)\cdots\hat{f}(t,\xi_p)\psi(\frac{\xi_1}{2^{k_1}})\cdots \psi(\frac{\xi_{p-1}}{2^{k_{p-1}}})d\xi_1\cdots d\xi_{p-1}.
\end{align*}
Let $|\xi|\in [2^j,2^{j+1}]$. We often abbreviate $2^a\ll 2^b$ by $a\ll b$ and similarly for $2^a\lesssim 2^b$ and $2^a\sim 2^b$ whenever $a,b\in \{j,k_1,\ldots,k_{p-1}\}.$\\
Now we estimate the above integral by dividing into several cases:
\begin{enumerate}[(i)] 
    \item All frequencies are bounded by $t^{-\frac{1}{n}},$ namely $2^j, 2^{k_1},\ldots,2^{k_{p-1}}\lesssim t^{-\frac{1}{n}}$. In this case, the integral can be estimated as:\begin{align*}
    &\left|\sum_{2^j,2^{k_1},\ldots,2^{k_{p-1}}\lesssim t^{-\frac{1}{n}}}i\xi \int e^{-itH}\hat{f}(t,\xi_1)\cdots\hat{f}(t,\xi_p)\psi(\frac{\xi_1}{2^{k_1}})\cdots \psi(\frac{\xi_{p-1}}{2^{k_{p-1}}})d\xi_1\cdots d\xi_{p-1}\right|\\\lesssim&\sum_{ 2^j,2^{k_1},\ldots,2^{k_{p-1}}\lesssim t^{-\frac{1}{n}}}2^j\epsilon_1^p 2^{k_1+\cdots+ k_{p-1}}\lesssim 2^j\epsilon_1^pt^{-\frac{p-1}{n}}\mathbf{1}_{t<2^{-nj}}\lesssim 2^j \epsilon t^{-\frac{n-1}{n}}\mathbf{1}_{t<2^{-nj}},
    \end{align*} so integrating in time gives \[\int_0^{2^{-nj}}2^j\epsilon t^{-\frac{n-1}{n}}dt\lesssim \epsilon.\]
    \item $2^j>t^{-\frac{1}{n}}$ and ${k_1},\ldots,{k_{p-1}}\lesssim j.$ In this case by Lemma 4.5 one has \begin{align*}
    I&:=\sum_{{k_1},\ldots,{k_{p-1}}\lesssim j}i\xi \int e^{-itH(\xi,\xi_1,\ldots,\xi_{p-1})}\hat{f}(t,\xi_1)\cdots\hat{f}(t,\xi_p)\psi(\frac{\xi_1}{2^{k_1}})\cdots \psi(\frac{\xi_{p-1}}{2^{k_{p-1}}})d\xi_1\cdots d\xi_{p-1}\\=&i\xi\int e^{-itH}\hat{f}(t,\xi_1)\cdots\hat{f}(t,\xi_p)\psi(\frac{\xi_1}{2^j})\cdots \psi(\frac{\xi_{p-1}}{2^j})d\xi_1\cdots d\xi_{p-1}\\=&i2^{pj}\int e^{-it2^{nj}H(\xi',\xi_1,\ldots,\xi_{p-1})}\xi'\hat{f}_{\lesssim j}(t,2^j\xi_1)\cdots\hat{f}_{\lesssim j}(t,2^j\xi_p')\chi(\frac{\xi_1}{C})\cdots \chi(\frac{\xi_{p-1}}{C})d\xi_1\cdots d\xi_{p-1}\\=&\frac{t^{-\frac{p-1}{2}}\xi}{|\xi|^{\frac{(p-1)(n-2)}{2}}}\sum_{\substack{0\leq j\leq p-1\\j\neq \frac{p-2}{2}}}c_je^{id_j t\xi^n}\hat{f}(t,-\xi_{p,j})^{j+1}\hat{f}(t,\xi_{p,j})^{p-j-1}\mathbf{1}_{|\xi|>t^{-\frac{1}{n}}}+2^{pj}O\left(\frac{\Vert\langle x\rangle^{2\rho}\hat{F}\Vert_{L_x^1}}{(2^{nj}t)^{\frac{p-1}{2}+\rho}}\right)
    \end{align*} where $\xi'=2^{-j}\xi$ and $\xi_p'=2^{-j}(\xi-\xi_1-\cdots-\xi_{p-1}),$ and $F(\xi_1,\ldots,\xi_{p-1})=\hat{f}_{\lesssim j}(t,2^j\xi_1)\cdots\hat{f}_{\lesssim j}(t,2^j\xi_p').$\\Also one may calculate \[\hat{F}(x_1,\ldots,x_{p-1})=2^{-pj}\int e^{-iz\xi}f_{\lesssim j}(\frac{z}{2^j})\prod_{l=1}^{p-1}f_{\lesssim j}(\frac{z-x_l}{2^j})dz,\] so that Minkowski's integral inequality gives \begin{align*}
    \Vert \hat{F}\Vert_{L^1}\lesssim& 2^{-pj}\int |f_{\lesssim j}(\frac{z}{2^j})|\prod_{l=1}^{p-1}\Vert f_{\lesssim j}(\frac{z-\cdot}{2^j})\Vert_{L^1_{x_l}} dz=\Vert f_{\lesssim j}\Vert_{L^1}^p,\\\Vert |x|^{2\rho}\hat{F}\Vert_{L^1}\lesssim & 2^{-pj}\int (|z|^{2\rho}+\sum_{l=1}^{p-1}|z-x_l|^{2\rho})|f_{\lesssim j}(\frac{z}{2^j})|\prod_{l=1}^{p-1}|f_{\lesssim j}(\frac{z-x_l}{2^j})|dzdx_1\cdots dx_{p-1}\\\lesssim& 2^{2\rho j}\Vert f_{\lesssim j}\Vert_{L^1}^{p-1} \Vert|x|^{2\rho}f_{\lesssim j}\Vert_{L^1}.
    \end{align*} Hence, finally one has 
    \begin{align*}
    &|I-\frac{t^{-\frac{p-1}{2}}\xi}{|\xi|^{\frac{(p-1)(n-2)}{2}}}\sum_{\substack{1\leq j\leq p\\j\neq \frac{p-2}{2}}}c_je^{id_j t\xi^n}\hat{f}(t,-\xi_{p,j})^{j+1}\hat{f}(t,\xi_{p,j})^{p-j-1}\mathbf{1}_{|\xi|>t^{-\frac{1}{n}}}|\\\lesssim&2^{pj}2^{-n(\frac{p-1}{2}+\rho)j}t^{-\frac{p-1}{2}-\rho}\left(\Vert f_{\lesssim j}\Vert_{L^1}^p+2^{2\rho j}\Vert f_{\lesssim j}\Vert_{L^1}^{p-1} \Vert|x|^{2\rho}f_{\lesssim j}\Vert_{L^1}\right)\mathbf{1}_{t>2^{-nj}}\\\lesssim& 2^{pj}(2^jt^{\frac{1}{n}})^{-\kappa}\epsilon_1^p\mathbf{1}_{t>2^{-nj}}\lesssim2^{-(\kappa-p)j}t^{-1-\frac{\kappa-p}{n}}\epsilon \mathbf{1}_{t>2^{-nj}},
    \end{align*} where $\kappa:=(n-1)\rho+(\frac{n}{2}-\frac{1}{4})(p-1)-\frac{1}{4}>p$ if $\rho$ is chosen sufficiently close to $\frac{1}{4}$. Hence integrating in time gives the bound $\lesssim\epsilon.$
    \item From now on, by symmetry we may only consider the case when $k_1$ is the largest among $k_1,\ldots,k_{p-1}.$ \\Here we consider the case $2^{k_1}>t^{-\frac{1}{n}}$ and ${k_1}\gg \text{others}.$ In this case, one has \[|\partial_{2}H|=|-n\xi_2^{n-1}+n(\xi-\xi_1-\cdots-\xi_n)^{n-1}|\sim 2^{(n-1)k_1},\quad \partial_2:=\partial_{\xi_2}.\] Moreover, the bounds on $\dfrac{1}{\partial_2H}$ are given as \[\left|\partial^\beta \left(\frac{1}{\partial_2H}\right)\right|\lesssim 2^{-(n-1)k_1}2^{-|\beta|k_1}\] for any multi-index $\beta$, which can be obtained via induction on $|\beta|$. This will be used to prove the following lemma:
\begin{lem}
    Let $m(\xi_1,\ldots,\xi_{p-1}):=\dfrac{\psi(2^{-k_1}\xi_1)\chi(2^{-k_1}\xi_2)\cdots\chi(2^{-k_1}\xi_{p-1})}{\partial_2H}.$ Then $m$ satisfies \[\Vert \int_{\mathbb{R}^{p-1}}m(\eta)e^{ix\cdot\eta}d\eta\Vert_{L_x^1}\lesssim 2^{-(n-1)k_1}.\]
    \end{lem}
\begin{proof}
    First observe that \[|\partial^\alpha m|\lesssim 2^{-(n-1)k_1}2^{-|\alpha|k_1}\] for any multi-index $\alpha$. Then one has
    \begin{align*}
    &\left||x|^{2p}\int_{\mathbb{R}^{p-1}}m(\eta)e^{ix\cdot\eta}d\eta\right|=\left|\int_{\mathbb{R}^{p-1}}[(\partial_1^2+\cdots\partial_{p-1}^2)^pm(\eta)]e^{ix\cdot\eta}d\eta\right|\\\lesssim& \int_{|\eta|\lesssim 2^{k_1}}2^{-(n-1)k_1}2^{-2pk_1}d\eta\lesssim 2^{-(n+p)k_1}.
    \end{align*}
    Moreover, a straightforward estimate gives \[\left|\int_{\mathbb{R}^{p-1}}m(\eta)e^{ix\cdot\eta}d\eta\right|\lesssim \int_{|\eta|\lesssim 2^{k_1}}2^{-(n-1)k_1}d\eta\lesssim 2^{-(n-p)k_1}.\] Therefore, one has \begin{align*}
    \Vert \int_{\mathbb{R}^{p-1}}m(\eta)e^{ix\cdot\eta}d\eta\Vert_{L_x^1}\lesssim&\int_{|x|\leq 2^{-k_1}}2^{-(n-p)k_1}dx+\int_{|x|> 2^{-k_1}}2^{-(n+p)k_1}|x|^{-2p}dx\\\lesssim &2^{-(n-1)k_1},
    \end{align*} which completes the proof.
\end{proof}
Now integration by parts gives
\begin{align*}
        &\sum_{\substack{{k_1}\gg \text{others}\\2^{k_1}>t^{-\frac{1}{n}}}}i\xi \int e^{-itH}\hat{f}(t,\xi_1)\cdots\hat{f}(t,\xi_p)\psi(\frac{\xi_1}{2^{k_1}})\cdots \psi(\frac{\xi_{p-1}}{2^{k_{p-1}}})d\xi_1\cdots d\xi_{p-1}\\=i\xi& \sum_{\substack{{k_1}\gg j\\2^{k_1}>t^{-\frac{1}{n}}}}\int e^{-itH}\hat{f}(t,\xi_1)\cdots\hat{f}(t,\xi_p)\psi(\frac{\xi_1}{2^{k_1}})\chi(\frac{C\xi_2}{2^{k_1}})\cdots \chi(\frac{C\xi_{p-1}}{2^{k_1}})d\xi_1\cdots d\xi_{p-1}\\=-i\xi& \sum_{\substack{{k_1}\gg j\\2^{k_1}>t^{-\frac{1}{n}}}}\int e^{-itH}\partial_2(\frac{1}{\partial_2H}\hat{f}(t,\xi_1)\cdots\hat{f}(t,\xi_p)\psi(\frac{\xi_1}{2^{k_1}})\chi(\frac{C\xi_2}{2^{k_1}})\cdots \chi(\frac{C\xi_{p-1}}{2^{k_1}}))d\xi_1\cdots d\xi_{p-1}\\=-\frac{\xi}{t}& \sum_{\substack{{k_1}\gg j\\2^{k_1}>t^{-\frac{1}{n}}}}\int e^{-itH}\frac{\partial_2^2H}{(\partial_2H)^2}\hat{f}_{\lesssim k_1}(t,\xi_1)\cdots\hat{f}_{\lesssim k_1}(t,\xi_p)\psi(\frac{\xi_1}{2^{k_1}})\chi(\frac{C\xi_2}{2^{k_1}})\cdots \chi(\frac{C\xi_{p-1}}{2^{k_1}})d\xi_1\cdots d\xi_{p-1}\\+\frac{\xi}{t}&\sum_{\substack{{k_1}\gg j\\2^{k_1}>t^{-\frac{1}{n}}}}\int e^{-itH}\frac{1}{\partial_2H}\hat{f}_{\lesssim k_1}(t,\xi_1)\partial_2\hat{f}_{\lesssim k_1}(t,\xi_2)\cdots\hat{f}_{\lesssim k_1}(t,\xi_p)\psi(\frac{\xi_1}{2^{k_1}})\chi(\frac{C\xi_2}{2^{k_1}})\cdots \chi(\frac{C\xi_{p-1}}{2^{k_1}}))d\xi_1\cdots d\xi_{p-1}\\-\frac{\xi}{t}&\sum_{\substack{{k_1}\gg j\\2^{k_1}>t^{-\frac{1}{n}}}}\int e^{-itH}\frac{1}{\partial_2H}\hat{f}_{\lesssim k_1}(t,\xi_1)\cdots\partial_2\hat{f}_{\lesssim k_1}(t,\xi_p)\psi(\frac{\xi_1}{2^{k_1}})\chi(\frac{C\xi_2}{2^{k_1}})\cdots \chi(\frac{C\xi_{p-1}}{2^{k_1}}))d\xi_1\cdots d\xi_{p-1}\\+\frac{\xi}{t}&\sum_{\substack{{k_1}\gg j\\2^{k_1}>t^{-\frac{1}{n}}}}\int e^{-itH}\frac{1}{\partial_2H}\hat{f}_{\lesssim k_1}(t,\xi_1)\cdots\hat{f}_{\lesssim k_1}(t,\xi_p)\psi(\frac{\xi_1}{2^{k_1}})\frac{C}{2^{k_1}}\chi'(\frac{C\xi_2}{2^{k_1}})\cdots \chi(\frac{C\xi_{p-1}}{2^{k_1}}))d\xi_1\cdots d\xi_{p-1}\\:=&\sum_{\substack{{k_1}\gg j\\2^{k_1}>t^{-\frac{1}{n}}}}(I_{k_1,1}+I_{k_1,2}+I_{k_1,3}+I_{k_1,4}).
    \end{align*}
    Therefore, by the Lemmas 4.6 and 4.8, one has \begin{align*}|I_{k_1,2}|+|I_{k_1,3}|\lesssim &\frac{2^j}{t2^{(n-1)k_1}}\Vert u_{\lesssim k_1}\Vert_{L^\infty}^{p-2}\Vert (xf)_{\lesssim k_1}\Vert_{L^2}\Vert f_{\lesssim k_1}\Vert_{L^2}\\\lesssim& \frac{2^j}{t^{1+\frac{p-5/2}{n}}2^{(n-3/2)k_1}}\epsilon_1^p\lesssim \frac{\epsilon2^j}{2^{(n-3/2)k_1}}t^{-2+\frac{5}{2n}}.\end{align*}Hence summing in $k_1$ and integrating in $t$ gives \begin{align*}
    \int_0^\infty \sum_{\substack{{k_1}\gg j\\2^{k_1}>t^{-\frac{1}{n}}}}|I_{k_1,2}|+|I_{k_1,3}|dt=\sum_{k_1\gg j}\int_{2^{-nk_1}}^\infty |I_{k_1,2}|+|I_{k_1,3}|dt\lesssim \sum_{k_1\gg j}\frac{\epsilon2^j}{2^{k_1}}\lesssim\epsilon.
    \end{align*} Since $I_{k_1,1}$ and $I_{k_1,4}$ can be estimated similarly, we omit the details.
    \item $2^{k_1}>t^{-\frac{1}{n}}$ and ${k_1}\sim {k_2}\gg j$, i.e. there are two or more largest frequencies. In this case, we may assume $k_1\sim k_2\gg \text{others},$ since the case when three or more frequencies are comparable to $k_1$ is similar. In this case one has \begin{align*}
    &\sum_{\substack{{k_1}\sim k_2\gg \text{others}\\2^{k_1}>t^{-\frac{1}{n}}}}i\xi \int e^{-itH}\hat{f}(t,\xi_1)\cdots\hat{f}(t,\xi_p)\psi(\frac{\xi_1}{2^{k_1}})\cdots \psi(\frac{\xi_{p-1}}{2^{k_{p-1}}})d\xi_1\cdots d\xi_{p-1}\\=&\sum_{\substack{{k_1}\gg j\\2^{k_1}>t^{-\frac{1}{n}}}}i\xi \int e^{-itH}\hat{f}_{\lesssim k_1}(t,\xi_1)\cdots\hat{f}_{\lesssim k_1}(t,\xi_p)\psi(\frac{\xi_1}{2^{k_1}})\tilde{\psi}(\frac{\xi_2}{2^{k_1}})\chi(\frac{C\xi_{3}}{2^{k_{1}}})\cdots \chi(\frac{C\xi_{p-1}}{2^{k_{1}}})d\xi_1\cdots d\xi_{p-1}\\=&\sum_{\substack{{k_1}\gg j\\2^{k_1}>t^{-\frac{1}{n}}}}i\xi2^{(p-1)k_1}\int e^{-i2^{nk_1}tH(\xi',\xi_1,\ldots,\xi_{p-1})}\hat{f}_{\lesssim k_1}(t,2^{k_1}\xi_1)\cdots\hat{f}_{\lesssim k_1}(t,2^{k_1}\xi_p')\\&\times\psi(\xi_1)\tilde{\psi}(\xi_2)\chi(C\xi_{3})\cdots \chi(C\xi_{p-1})d\xi_1\cdots d\xi_{p-1}\\:=&\sum_{\substack{{k_1}\gg j\\2^{k_1}>t^{-\frac{1}{n}}}}J_{k_1},
    \end{align*} where $\xi$ and $\xi_p'$ are the same as in the case (ii), and $\tilde{\psi}$ is a bump function supported on $|\xi|\sim 1$ but has a wider support than $\psi$. Now again Lemma 4.5 gives \begin{align*}
    |J_{k_1}|\lesssim 2^j2^{(p-1)k_1}\frac{\Vert \langle x\rangle^{\rho}\hat{F}\Vert_{L^1_x}}{(2^{nk_1}t)^{\frac{p-1}{2}+\rho}}
    \end{align*} with $F:=\hat{f}_{\lesssim k_1}(t,2^{k_1}\xi_1)\cdots\hat{f}_{\lesssim k_1}(t,2^{k_1}\xi_p')$. Now a similar calculation to the case (ii) gives the bound \[|J_{k_1}|\lesssim \frac{2^j}{2^{k_1}}2^{-(\kappa-p)k_1}t^{-1-\frac{\kappa-p}{n}}\epsilon,\] where $\kappa$ is the same as in (ii). Hence summation on $k_1$ and integration on $t$ gives
    \[\int_0^\infty \sum_{\substack{{k_1}\gg j\\2^{k_1}>t^{-\frac{1}{n}}}}|J_{k_1}|\lesssim \sum_{k_1\gg j}\frac{\epsilon 2^j}{2^{k_1}}2^{-(\kappa-p)k_1}\int_{2^{-nk_1}}^\infty t^{-1-\frac{\kappa-p}{n}}dt \lesssim  \sum_{k_1\gg j}\frac{\epsilon 2^j}{2^{k_1}}\lesssim \epsilon\] and completes the proof.
\end{enumerate}
\printbibliography[
heading=bibintoc,
title={References}
]
\end{document}